\documentclass{amsart}
\RequirePackage[utf8]{inputenc}
\usepackage{amsmath,amsthm,amsfonts,amssymb,amscd,amsbsy,dsfont}
\usepackage{graphicx}
\usepackage{hyperref}

\newcommand{\dd}{\mathrm{d}}

\newcommand{\Ric}{\operatorname{Ric}}
\newcommand{\Z}{\mathds Z}

\newcommand{\R}{\mathds R}
\newcommand{\C}{\mathds C}
\newcommand{\Hr}{\mathds H}

\renewcommand{\S}{\mathsf{S}}
\newcommand{\SO}{\mathsf{SO}}
\renewcommand{\O}{\mathsf O}
\newcommand{\SU}{\mathsf{SU}}

\newcommand{\Sp}{\mathsf{Sp}}

\newcommand{\G}{\mathsf{G}}
\newcommand{\K}{\mathsf{K}}
\renewcommand{\H}{\mathsf{H}}
\newcommand{\N}{\mathsf{N}}

\newcommand{\g}{\mathrm g}

\newcommand{\diag}{\operatorname{diag}}

\newcommand{\Ad}{\operatorname{Ad}}
\newcommand{\ggz}{\g_{\rm GZ}}

\newcommand{\zed}{\mathds{Z}}

\allowdisplaybreaks

\hypersetup{
    pdftoolbar=true,
    pdfmenubar=true,
    pdffitwindow=false,
    pdfstartview={FitH},
    pdftitle={Four-dimensional cohomogeneity one Ricci flow and nonnegative sectional curvature},
    pdfauthor={Renato G. Bettiol and Anusha M. Krishnan},
    pdfsubject={},
    pdfkeywords={}
    pdfnewwindow=true,
    colorlinks=true, 
    linkcolor=blue,
    citecolor=blue,
    urlcolor=black,
}

\newtheorem*{theorem*}{Theorem}
\newtheorem{theorem}{Theorem}
\newtheorem{proposition}[theorem]{Proposition}

\newtheorem*{mainthm}{\sc Theorem}
\newtheorem*{maincor}{\sc Corollary}

\theoremstyle{definition}

\theoremstyle{remark}
\newtheorem{remark}[theorem]{Remark}

\title[Four-dimensional cohomogeneity one Ricci flow]{Four-dimensional cohomogeneity one Ricci flow and nonnegative sectional curvature}

\author[R. G. Bettiol]{Renato G. Bettiol}
\author[A. M. Krishnan]{Anusha M. Krishnan}

\address{\begin{tabular}{l}
University of Pennsylvania \\
Department of Mathematics \\
209 South 33rd St \\
Philadelphia, PA, 19104-6395, USA \\[0.3cm]
\emph{E-mail address}: {\tt rbettiol@math.upenn.edu}\\
\emph{E-mail address}: {\tt anushakr@math.upenn.edu}
\end{tabular}
}

\allowdisplaybreaks
\numberwithin{equation}{section}
\numberwithin{theorem}{section}


\date{\today}

\begin{document}
\begin{abstract}
We exhibit the first examples of closed $4$-manifolds with nonnegative sectional curvature that lose this property when evolved via Ricci flow.
\end{abstract}

\maketitle

\section{Introduction}
Several great successes in Geometric Analysis continue to be achieved through \emph{Ricci flow}, a technique introduced by Hamilton~\cite{hamilton-original} around $35$ years ago. This is a way of evolving Riemannian metrics $\g$ on a manifold $M$ via the geometric PDE 
\begin{equation}\label{eq:RF}
\frac{\partial \g}{\partial t} = -2 \Ric_{\g}
\end{equation}
where $\Ric_\g$ is the Ricci tensor of $\g$. The underlying theme in applications of this technique is that the Ricci flow of Riemannian metrics, similarly to the heat flow of temperature distributions and other diffusion processes, should have regularizing properties that eventually evolve a metric to some \emph{canonical} or \emph{best} metric on $M$, whose existence allows to draw topological conclusions about $M$.

A fundamental step to carry out geometric applications is understanding the behavior of curvature conditions along the flow.
In his seminal paper~\cite{hamilton-original}, Hamilton proved that \eqref{eq:RF} preserves nonnegative Ricci curvature $(\Ric\geq 0)$ and nonnegative sectional curvature $(\sec\geq0)$ on closed $3$-manifolds, as well as nonnegative scalar curvature in closed manifolds of all dimensions.
Hamilton also proved that positive-semidefiniteness of the curvature operator is preserved in closed manifolds of all dimensions~\cite{hamilton-4d}, and nonnegative isotropic curvature is preserved on closed $4$-manifolds~\cite{hamilton-4d-pic}. Independently, Brendle and Schoen~\cite{bs} and Nguyen~\cite{nguyen} generalized the latter to closed manifolds of all dimensions. An elegant and unified approach to proving invariance of all the above curvature nonnegativity conditions under the Ricci flow was developed by Wilking~\cite{wilking-lie}.

On the opposite side, several curvature conditions have been shown not to be preserved by \eqref{eq:RF}. For instance, M\'aximo~\cite{maximo-rf,maximo2} constructed K\"ahler metrics on $4$-manifolds with $\Ric\geq0$ or $\Ric>0$ (but without $\sec\geq0$) that evolve to metrics with mixed Ricci curvature.
B\"{o}hm and Wilking~\cite{bw-pi1-ricci} exhibited homogeneous metrics with $\sec>0$ that develop mixed Ricci curvature in dimension $12$, and mixed sectional curvature in dimension $6$. The latter behavior was shown to be generic among homogeneous metrics on these manifolds by Abiev and Nikonorov~\cite{abiev-nikonorov}. Finally, \emph{noncompact} examples of complete manifolds with $\sec\geq0$ that develop mixed sectional curvature in all dimensions $\geq4$ were found by Ni~\cite{ni}. Nevertheless, the existence of closed manifolds exhibiting such behavior in dimensions $4$ and $5$ remained unsettled.
The main result of this paper is that many such examples exist:

\begin{mainthm}
There exist metrics with $\sec\geq0$ on $S^4$, $\C P^2$, $S^2\times S^2$, and $\C P^2 \# \overline{\C P^2}$ that immediately lose this property when evolved via Ricci~flow.
\end{mainthm}

By taking products of the above manifolds with spheres, one easily concludes:

\begin{maincor}
Ricci flow does not preserve $\sec\geq0$ on closed manifolds of any dimension~$\geq4$.
\end{maincor}

We remark that the $4$-manifolds listed in the Theorem, together with $\C P^2 \# \C P^2$, are the only closed simply-connected $4$-manifolds currently known to admit metrics with $\sec\geq0$. Conjecturally, this list is complete.
Furthermore, the manifolds in the Theorem are the only closed simply-connected $4$-manifolds that carry a \emph{cohomogeneity one} action, i.e., an isometric action whose orbit space is $1$-dimensional. The metrics with $\sec\geq0$ used as initial data were introduced by Grove and Ziller~\cite{grove-ziller-annals} and are invariant under these large isometry groups, hence so are their Ricci flow evolution.
Exploiting the fact that isometries are preserved, one may translate the Ricci flow equation \eqref{eq:RF} into a more accessible system of coupled PDEs in only $2$ variables (one for time and one for space), see Proposition~\ref{prop:RFequations}. This allows us to explicitly compute the first variation of the sectional curvature of certain initially flat planes and determine it is negative, hence the manifold immediately acquires some negatively curved planes under the flow.
Similar cohomogeneity one frameworks were previously employed by B\"ohm~\cite{bohm-inventiones} and Dancer and Wang~\cite{dancer-wang} to construct Einstein metrics and Ricci solitons, by Pulemotov~\cite{pulemotov} to study Ricci flow on manifolds with boundary, and implicitly in several other recent works including \cite{ak,warpedBerger,ls}.

It is our hope that this unifying viewpoint of \emph{cohomogeneity one Ricci flow} will be more systematically studied in the future, mirroring the ongoing study of \emph{homogeneous Ricci flow} pioneered by Lauret~\cite{lauret1}, B\"ohm~\cite{bohm}, B\"ohm and Lafuente~\cite{bl}, and others, in which the Ricci flow equations \eqref{eq:RF} reduce to an ODE. In some sense, this is the next step in a symmetry program approach to understanding Ricci flow.

There are several general issues to be addressed, e.g., determining under which conditions the more restrictive \emph{diagonal} cohomogeneity one Ansatz is preserved (see Proposition~\ref{prop:diagonalpreserved} for the particular case at hand in this paper, and Remark~\ref{rem:diagonalpreservation}). Similar issues were confronted by Lauret and Will~\cite{lauret-will2} in the case of Ricci flow on Lie groups, and by Dammerman~\cite{diagonalize-cohom1} for cohomogeneity one Einstein manifolds. Other natural future directions include investigating the long-term behavior of cohomogeneity one Ricci flow and the types of singularities that it may develop.

This paper is organized as follows. A recollection of facts about cohomogeneity one manifolds is given in Section~\ref{sec:cohom1}. Section~\ref{sec:examples} has a detailed account of the $4$-dimensional examples. The behavior of these manifolds and their sectional curvature under Ricci flow is addressed in Section~\ref{sec:flow}, where the Theorem is proved.

\subsection*{Acknowledgements} We thank Wolfgang Ziller for many helpful discussions, and Dan Knopf and Ricardo Mendes for comments and suggestions.

\section{Cohomogeneity one manifolds}\label{sec:cohom1}

In this section, we briefly review basic aspects of cohomogeneity one manifolds. The simply-connected $4$-dimensional examples and relevant nonnegatively curved metrics are described in the next section. For more details, we refer to~\cite{mybook,grove-ziller-annals,gz-inventiones,ziller-coh1survey}.

\subsection{Cohomogeneity one structure}
A group $\G$ acting isometrically on a Riemannian manifold $(M,\g)$ is said to act with \emph{cohomogeneity one} if the orbit space $M/\G$ is one-dimensional. If $M$ is compact, then $M/\G$ with the induced orbital distance is either isometric to a circle $S^1$ or to a closed interval $[0,L]$. We shall focus on the latter case, in which $M$ is topologically more interesting.
For each $r\in M/\G$, $0<r<L$, the preimage $\pi^{-1}(r)\subset M$ is a \emph{principal orbit}, that is, a (codimension $1$) hypersurface in $M$. The preimages of the endpoints, $B_-=\pi^{-1}(0)$ and $B_+=\pi^{-1}(L)$, are nonprincipal orbits, which are called \emph{singular} (if the codimension is $\geq2$), or \emph{exceptional} otherwise. Nonprincipal orbits on simply-connected cohomogeneity one manifolds are always singular.

Pick a point $x_{-}\in B_{-}$ and consider a minimal geodesic $\gamma(r)$ in $M$ joining $x_{-}$ to $B_{+}$, meeting it at $x_{+}=\gamma(L)$; that is, $\gamma$ is a horizontal lift of the interval $[0,L]$ to $M$. Denote by $\K_{\pm}$ the isotropy group at $x_{\pm}$, and by $\H$ the isotropy at an interior point $\gamma(r)$. The principal isotropy $\H$ remains the same group at all $\gamma(r)$, for $r\in (0,L)$, and is a subgroup of $\K_{\pm}$. This gives a decomposition of $M$ as the union of orbits $\G(\gamma(r))$, $0\leq r\leq L$, each of which is a homogeneous space; $\G/\H$ at interior points $r\in (0,L)$, and $\G/\K_{\pm}$ at the endpoints $r=0$ and $r=L$.

By the Slice Theorem, the tubular neighborhoods $D(B_{-}) = \pi^{-1}\big(\big[0,\frac{L}{2}\big]\big)$ and $D(B_{+}) = \pi^{-1}\big(\big[\frac{L}{2},L\big]\big)$ of the singular orbits are disk bundles over $B_-$ and $B_+$. Let $D^{l_{\pm}+1}$ be the normal disks to $B_{\pm}$ at $x_{\pm}$, so that $l_\pm+1$ is the codimension of $B_\pm$. Then $\K_{\pm}$ acts transitively on the boundary $\partial D^{l_{\pm}+1}$, with isotropy $\H$, so $\partial D^{l_{\pm}+1} = S^{l_{\pm}} = \K_{\pm}/\H$, and the $\K_{\pm}$-action on $\partial D^{l_{\pm}+1}$ extends to a $\K_\pm$-action on all of $D^{l_{\pm}+1}$. Moreover, there are equivariant diffeomorphisms of the disk bundles:
\begin{equation*}
 D(B_{\pm}) \cong\G\times_{\K_{\pm}}D^{l_{\pm}+1}.
\end{equation*}
The manifold $M$ is the union of the above disk bundles, glued along their common boundary $\G/\H$. One associates to such a manifold $M$ the \emph{group diagram} $\H\subset\{\K_-,\K_+\}\subset \G$. Conversely, given a group diagram $\H\subset \{\K_{-}, \K_{+}\}\subset \G$, where $\K_{\pm}/\H$ are spheres, there exists a cohomogeneity one manifold $M$ given as the union of the above disk bundles.

The full isometry group of a cohomogeneity one manifold $(M,\g)$ is often strictly larger than the group $\G$ that acts with cohomogeneity one, and we make frequent use of this fact in what follows.

\subsection{Cohomogeneity one metrics}
Since $\G$ acts on $(M,\g)$ by isometries, the metric $\g$ is completely determined by its restriction to the geodesic $\gamma(r)$, which meets all orbits orthogonally. Furthermore, it suffices to determine $\g$ on the open and dense subset $M\setminus B_\pm$, corresponding to $\gamma(r)$, $r\neq 0,L$. On this subset, we write:
\begin{equation}\label{eq:cohom1metric}
\g = \dd r^2 + \g_r, \qquad 0<r<L,
\end{equation}
where $\g_r$ is a $1$-parameter family of homogeneous metrics on $\G/\H$. Conversely, in order to define a cohomogeneity one metric $\g$ by means of the above equation, certain \emph{smoothness conditions} must be fulfilled at the endpoints $r=0$ and $r=L$.

We henceforth only consider cohomogeneity one metrics that are \emph{diagonal}, in a sense slightly stronger than in \cite{gz-inventiones}. More precisely, let $\{v_i\}$ be a basis of the Lie algebra of $\G$ which is adapted to the inclusions $\H\subset\{\K_-,\K_+\}\subset\G$ and orthonormal with respect to a fixed bi-invariant metric. Consider the induced action fields $X_i(r)=\frac{\dd}{\dd s}\exp(s\, v_i)\cdot\gamma(r)\big|_{s=0}$. A metric \eqref{eq:cohom1metric} is \emph{diagonal} if it satisfies
\begin{equation*}
\g_r\big(X_i(r),X_j(r)\big)=f_i(r)^2\delta_{ij},
\end{equation*}
that is, $\g_r$ is the diagonal matrix $\diag(f_1^2,\cdots,f_k^2)$ in the basis $\{X_i\}$, where $k=\dim M-1$.
Note that $f_i(r)$ is hence the length of the Killing field $X_i(r)$, and this Killing field vanishes at $r=0$ or $r=L$ if and only if $v_i$ belongs to the Lie algebra of the corresponding isotropy subgroup $\K_\pm$.
In this situation, the smoothness conditions translate into conditions on the Taylor series of $f_i(r)$ at $r=0$ and $r=L$.
Details on how to compute such smoothness conditions in terms of the algebraic data in the cohomogeneity one group diagram can be found in the forthcoming paper~\cite{verdiani-ziller-new}, see also \cite[Appendix]{p2} and \cite[Sec.\ 2]{gz-inventiones}.

\begin{remark}
Not all cohomogeneity one manifolds admit diagonal metrics in the above sense. A sufficient condition for the existence of such metrics is for the isotropy representation of $\H$ to split as a sum of pairwise inequivalent representations.
\end{remark}

\section{\texorpdfstring{On the $4$-dimensional examples}{On the 4-dimensional examples}}\label{sec:examples}
The only closed simply-con\-nected $4$-manifolds that admit cohomogeneity one structures are $S^4$,  $\C P^2$, $S^2\times S^2$, and $\C P^2 \# \overline{\C P^2}$, see \cite{parker-classfn}. We now list their group diagrams, corresponding to the cohomogeneity one actions that we use to describe metrics on these manifolds:

\begin{table}[ht]
 \begin{center}
  \begin{tabular}{ll}
    \hline\noalign{\smallskip}
      $M$ & $\H\subset \{\K_{-}, \K_{+}\}\subset \G$\\ 
       \hline\noalign{\medskip}
      $S^4$ & $\S(\O(1)\O(1)\O(1))\subset \{\S(\O(2)\O(1)), \S(\O(1)\O(2))\}\subset \SO(3)$\\ 
       \noalign{\smallskip}
      $\C P^2$ & $\zed_2 = \left\langle \diag(-1,-1,1)\right\rangle \subset \{\S(\O(1)\O(2)), \SO(2)_{1,2}\}\subset \SO(3)$ 
       \\ \noalign{\smallskip}
      $S^2\times S^2$ & $\zed_n=\left\langle e^{2\pi i/n}\right\rangle\subset\big\{ \{e^{i\theta}\}, \{e^{i\theta}\}\big\}\subset \Sp(1)$, \quad $n$ even
     \\  \noalign{\smallskip}
      $\C P^2 \# \overline{\C P^2}$ & $\zed_n=\left\langle e^{2\pi i/n}\right\rangle\subset\big\{ \{e^{i\theta}\}, \{e^{i\theta}\}\big\}\subset \Sp(1)$, \quad $n$ odd
      \\ \noalign{\smallskip}\hline
  \end{tabular}
 \end{center}
 \caption{Group diagrams for cohomogeneity one $4$-manifolds}
 \vspace{-0.5cm}
\end{table}

In the above, $\SO(2)_{1,2}$ is the upper block diagonal embedding of $\SO(2)$ in $\SO(3)$, and $\Sp(1)\cong S^3\subset\Hr$ is identified with the group of unit quaternions.

Note that the only groups $\G$ above are $\SO(3)$ and $\Sp(1)$, which have the same Lie algebra $\mathfrak{g}\cong\mathfrak{su}(2)$.
Moreover, the groups $\K_{\pm}$ consist of finitely many copies of $\SO(2)\cong S^1$, and the principal isotropy group $\H$ is finite, so its Lie algebra is trivial. In particular, on the regular part $M\setminus B_\pm$, there are $3$ linearly independent Killing vector fields $X_1$, $X_2$, and $X_3$, which are action fields corresponding to a basis of~$\mathfrak g$.
More precisely, $X_i(p)=\frac{\dd}{\dd s}\exp(s\, v_i)\cdot p\big|_{s=0}$, where $\{v_i\}$ is the basis $\{I,J,K\}$ in the case of $\Sp(1)$, and $\{E_{23},E_{31},E_{12}\}$ in the case of $\SO(3)$, where $E_{jk}$ is the skew-symmetric matrix with a $+1$ in the $(j,k)$ entry, a $-1$ in the $(k,j)$ entry, and zeros elsewhere.
Thus, fixing a horizontal geodesic $\gamma(r)$, a diagonal metric \eqref{eq:cohom1metric} on $M$ is of the form
\begin{equation}\label{eq:diagonalmetricstationary}
\g = \dd r^2 + \varphi(r)^2 \dd x_1^2 + \psi(r)^2 \dd x_2^2 + \xi(r)^2 \dd x_3^2, \qquad 0<r<L,
\end{equation}
where $\dd x_i$ is the $1$-form dual to $X_i$. The singular orbits $B_\pm=\G/\K_\pm$ in all above examples are $2$-dimensional, which means that only one of the functions $\varphi$, $\psi$, and $\xi$, vanishes at each of the endpoints $r=0$ and $r=L$. Since the codimension of $B_\pm$ is equal to $2$, by the work of Grove and Ziller~\cite{grove-ziller-annals}, these manifolds support $\G$-invariant diagonal metrics $\ggz$ with $\sec\geq0$. These are the nonnegatively curved metrics used to prove the Theorem.

We now discuss some details about these metrics, following \cite[Sec.\ 2]{ziller-coh1survey} and \cite{mybook}. Some features common to all of them (originating from the gluing in the Grove-Ziller construction), are the presence of flat planes at all points, including planes along $\gamma(r)$ that contain the tangent direction $\gamma'(r)$, see also Remark~\ref{rem:final}. Moreover, the functions among $\varphi$, $\psi$, and $\xi$ that do not vanish at the endpoint corresponding to a singular orbit $B_\pm$ are \emph{equal and constant} in a neighborhood of that endpoint, while the remaining functions vanish with \emph{nonvanishing first derivative}. 
Finally, there are \emph{sufficiently many isometries} to ensure that the Ansatz \eqref{eq:diagonalmetricstationary} is preserved along the Ricci flow. These features are key in the proof of the Theorem.

\subsection{\texorpdfstring{$S^4$}{Sphere}}\label{subsec:s4}
The $\SO(3)$-action on $S^4$ is the restriction to the unit sphere of the action by conjugation on the space $V$ of symmetric traceless $3\times 3$ real matrices. The singular orbits $B_\pm$ are Veronese embeddings of $\R P^2$ formed by matrices with $2$ equal eigenvalues of the same sign; while principal orbits are diffeomorphic to the real flag manifold $W^3=S^3/(\Z_2\oplus\Z_2)$ and formed by generic matrices in $V$. The horizontal geodesic joining $x_-=\tfrac{1}{\sqrt6}\diag(1,1,-2)\in B_-$ to $x_+=\tfrac{1}{\sqrt6}\diag(2,-1,-1)\in B_+$~is
\begin{equation*}
\gamma(r)=\diag\left(\tfrac{\cos r}{\sqrt6}+\tfrac{\sin r}{\sqrt2},\, \tfrac{\cos r}{\sqrt6}-\tfrac{\sin r}{\sqrt2},\, -\tfrac{2\cos r}{\sqrt6}\right)\in V, \qquad 0<r<\tfrac\pi3.
\end{equation*}
In this description, the round metric on $S^4$ takes the form \eqref{eq:diagonalmetricstationary} where
\begin{equation}\label{eq:roundsphere}
\varphi(r)=2\sin r, \;\; \psi(r)=\sqrt3\cos r+\sin r, \;\; \xi(r)=\sqrt3\cos r-\sin r.
\end{equation}
The metric $\ggz$ is also of the form \eqref{eq:diagonalmetricstationary}, with functions $\varphi$, $\psi$, and $\xi$ that satisfy the same smoothness conditions as the above at $r=0$ and $r=\tfrac\pi3$. However, they are constant away from a neighborhood of the vanishing locus (see Figure~\ref{fig:graphs}).
\begin{figure}[ht]
\begin{center}
\includegraphics[scale=0.48]{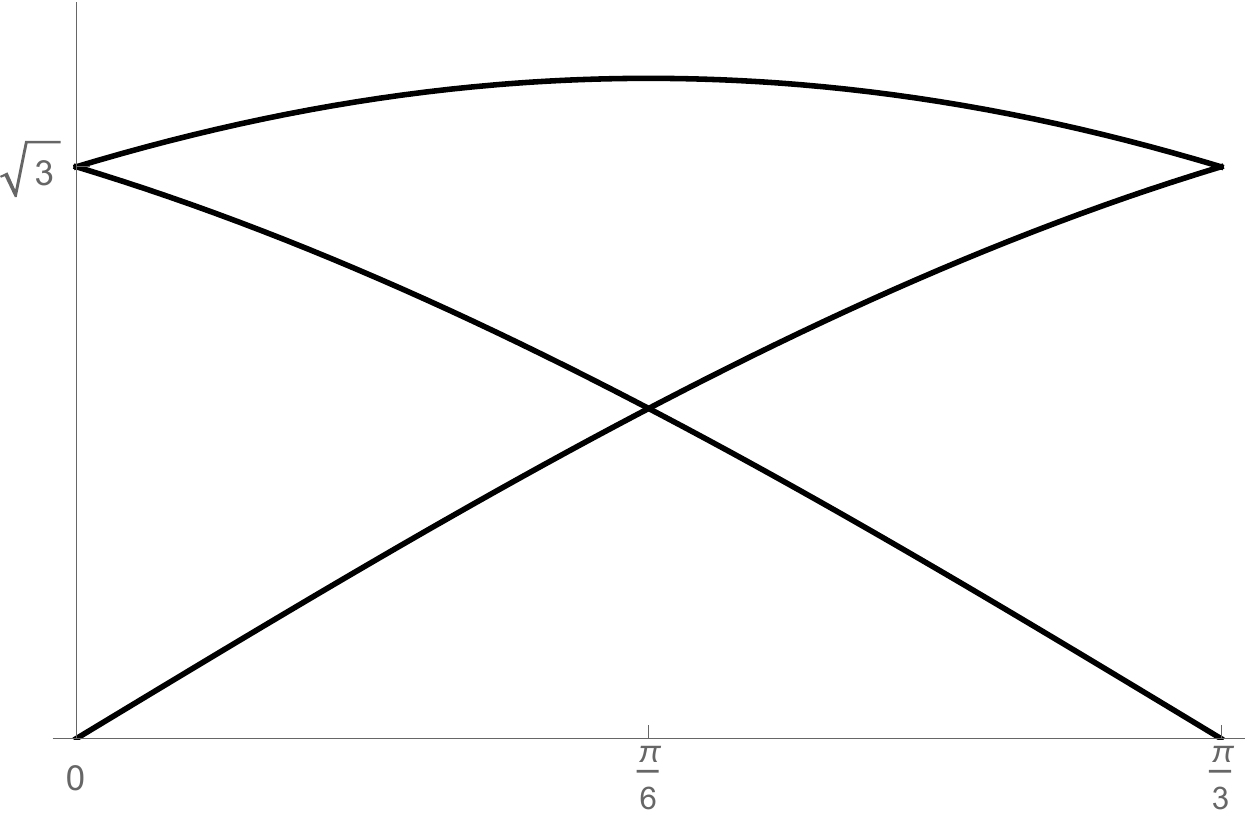}
\includegraphics[scale=0.48]{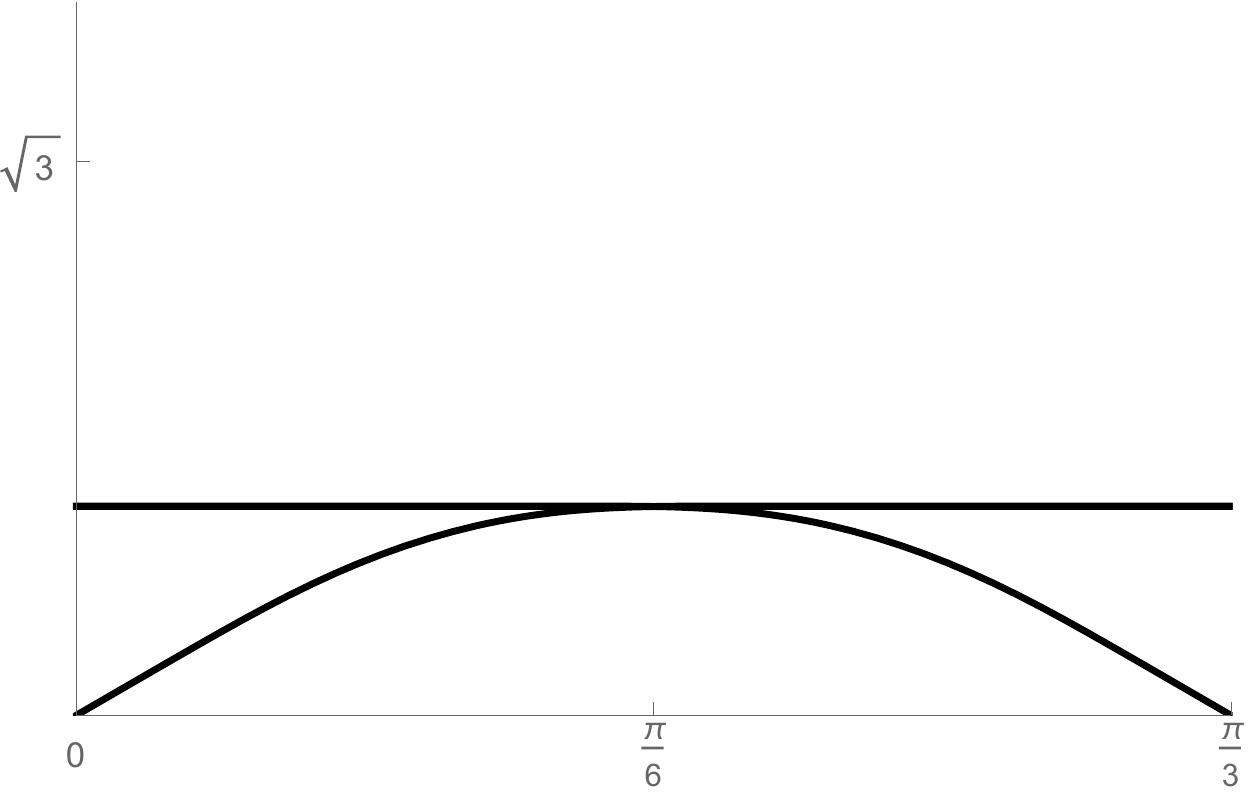}
\end{center}
\caption{The functions $\varphi$, $\psi$, and $\xi$ corresponding to the round metric (left) and to the Grove-Ziller metric $\ggz$ (right).}\label{fig:graphs}
\end{figure}

More generally, given any $\SO(3)$-invariant metric $\g$ on $S^4$, there are isometries given by $h_i\in \H$,
\begin{equation*}
\begin{aligned}
h_1&=\diag(1, -1, -1),\\
h_2&=\diag(-1, 1, -1), \\
h_3&=\diag(-1, -1, 1), 
\end{aligned}
\end{equation*}
that fix each point $\gamma(r)$ and such that $\dd h_i(\gamma(r))\colon T_{\gamma(r)}S^4\to T_{\gamma(r)}S^4$ are respectively
\begin{equation}\label{eq:dhs}
\begin{aligned}
\dd h_1(\gamma(r))&=\diag(1,1, -1, -1),\\
\dd h_2(\gamma(r))&=\diag(1,-1, 1, -1),\\
\dd h_3(\gamma(r))&=\diag(1,-1, -1, 1),
\end{aligned}
\end{equation}
with respect to the frame $\left\{\frac{\partial}{\partial r},X_1,X_2,X_3\right\}$ at $\gamma(r)$. Thus, $\g$ must be of the form \eqref{eq:diagonalmetricstationary}, i.e., $\left\{\frac{\partial}{\partial r},X_1,X_2,X_3\right\}$ is a \emph{$\g$-orthogonal frame} along $\gamma(r)$. Indeed, for $i\neq j$,
\begin{equation}\label{eq:isometry-trick}
\begin{aligned}
&\g(X_i,X_j)=\g\big(\dd h_i (X_i), \dd h_i (X_j)\big)=-\g(X_i,X_j)\\
&\g\big(\tfrac{\partial}{\partial r},X_j\big)=\g\big(\dd h_i\big(\tfrac{\partial}{\partial r}\big),\dd h_i \big(X_j\big)\big)=-\g\big(\tfrac{\partial}{\partial r},X_j\big).
\end{aligned}
\end{equation}

\begin{remark}
An alternative way of verifying the above claim is that $\gamma$ is a component of the fixed point set of the discrete group $\N(\H)/\H$, which consists of totally geodesic submanifolds. Thus, $\gamma$ is a horizontal geodesic (up to reparametrization) with respect to any metric $\g$ invariant under these isometries. Moreover, the vertical part $\g_r$ of such a metric $\g$ must be diagonal with respect to the above frame as it corresponds to an $\Ad(\H)$-invariant tensor on $\mathfrak{su}(2)$, which decomposes as the direct sum of $3$ inequivalent $1$-dimensional representations spanned by the $X_i$.
\end{remark}

\subsection{\texorpdfstring{$\C P^2$}{Complex projective plane}}\label{subsec:cp2}
The $\SO(3)$-action on $\C P^2$ is obtained as the subaction of the transitive $\SU(3)$-action. The singular orbit $B_-$ is the totally real $\R P^2\subset\C P^2$, and $B_+\cong S^2$ is the quadric $\big\{[z_0:z_1:z_2]\in\C P^2:\sum_j z_j^2=0\big\}$. The horizontal geodesic joining $x_-=[1:0:0]\in B_-$ to $x_+=\left[\tfrac{1}{\sqrt2}:\tfrac{i}{\sqrt{2}}:0\right]\in B_+$ is
\begin{equation*}
\gamma(r)=[\cos r:i\,\sin r:0], \qquad 0<r<\tfrac\pi4.
\end{equation*}
In this description, the Fubini-Study metric on $\C P^2$ takes the form \eqref{eq:diagonalmetricstationary} where
\begin{equation}\label{eq:fubinistudy}
\varphi(r)=\sin r, \;\; \psi(r)=\cos 2r, \;\; \xi(r)=\cos r,
\end{equation}
while $\ggz$ corresponds to functions $\varphi$, $\psi$, and $\xi$ that are qualitatively similar to those in the previous example.

Consider the complex conjugation map
\begin{equation}\label{eq:c}
c\colon\C P^2\to\C P^2, \qquad c\big([z_0:z_1:z_2]\big)=[\overline{z_0}:\overline{z_1}:\overline{z_2}],
\end{equation}
which clearly commutes with the $\SO(3)$-action and is an involution with fixed point set $B_-$. It is easy to show that $\phi=g\circ c$, where $g=\diag(1,-1,-1)\in\SO(3)$, is a diffeomorphism that fixes the above geodesic $\gamma(r)$ pointwise and whose linearization at any such point is the linear transformation on $T_{\gamma(r)}\C P^2$ with matrix $\phi_*=\diag(1,1, -1, -1)$ with respect to the frame $\left\{\tfrac{\partial}{\partial r},X_1,X_2,X_3\right\}$. In particular, this linear transformation is orthogonal with respect to any metric of the form \eqref{eq:diagonalmetricstationary}, including $\ggz$. It thus follows that $c$, and hence $\phi=g\circ c$, are isometries of $(\C P^2,\ggz)$. Indeed, given any $p\in\C P^2$, there exists $g_p\in\SO(3)$ such that $g_p\cdot p$ lies in $\gamma$, and hence one may write $c(p)=(gg_p)^{-1}gg_p\cdot c(p)=(gg_p)^{-1}g\cdot c(g_p\cdot p)$ as a composition of diffeomorphisms whose linearization is isometric.

We claim that if $\g$ is any $\SO(3)$-invariant Riemannian metric on $\C P^2$ such that $\phi$ is an isometry, then $\left\{\tfrac{\partial}{\partial r},X_1,X_2,X_3\right\}$ is $\g$-orthogonal and hence $\g$ must also be of the form \eqref{eq:diagonalmetricstationary}. Indeed, using $\phi$ in conjunction with $\diag(-1,-1,1)\in\H$, one can produce sufficiently many isometries of $(\C P^2,\g)$ that fix each point $\gamma(r)$ and act on $T_{\gamma(r)}\C P^2$ just as \eqref{eq:dhs}, so that an argument analogous to \eqref{eq:isometry-trick} may be carried~out.

\subsection{\texorpdfstring{$S^2\times S^2$ and $\C P^2\#\overline{\C P^2}$}{Sphere bundles over the sphere}}\label{subsec:mn}
The $\Sp(1)$-actions on $S^2\times S^2$ and $\C P^2\#\overline{\C P^2}$ are induced by quaternionic left-multiplication on the first factor of $S^3\times S^2\subset\Hr\oplus \C\oplus\R$ after taking the quotient by the diagonal circle action $e^{i\theta}\cdot (q,z,x)=\left(q\,e^{i\theta},z\,e^{in\theta},x\right)$. The orbit space $M_n=(S^3\times S^2)/S^1$ of this circle action is diffeomorphic to $S^2\times S^2$ if $n$ is even, and to $\C P^2\#\overline{\C P^2}$ if $n$ is odd. The singular orbits $B_\pm$ are both diffeomorphic to $S^2$, and lift to $S^3\times\{\pm N\}\subset S^3\times S^2$ where $N=\left(0,\tfrac12\right)\in S^2\left(\tfrac12\right)\subset\C\oplus\R$ is the North Pole, while principal orbits are diffeomorphic to the Lens space $S^3/\Z_n$. The horizontal geodesic joining $x_-=\left[1,0,-\tfrac12\right]$ to $x_+=\left[1,0,\tfrac12\right]$ is
\begin{equation*}
\gamma(r)=\left[1,\tfrac12\sin 2r,-\tfrac12\cos 2r\right]\in M_n, \qquad 0<r<\tfrac{\pi}{2},
\end{equation*}
where brackets indicate the coordinates induced by $\Hr\oplus\C\oplus\R$ in the quotient space. 
Similarly to the previous examples, in this description, the metric $\ggz$ on $M_n$ is of the form \eqref{eq:diagonalmetricstationary} with $\varphi$, $\psi$, and $\xi$ satisfying analogous properties.

Consider the involutions given by conjugation by $j,k\in\Sp(1)$,
\begin{equation}\label{eq:phis}
\phi_j,\phi_k\colon M_n\to M_n, \;\; \phi_j\left(\left[q,z,x\right]\right)=[-j\,q\,j,z,x],\;\,  \phi_k\left(\left[q,z,x\right]\right)=[-k\,q\,k,z,x].
\end{equation}
It is easy to see that the above maps are well-defined diffeomorphisms that leave invariant the $\Sp(1)$-orbits and act on them via conjugation, that is, the restrictions of $\phi_j$ and $\phi_k$ to $\G(\gamma(r))\cong\G/\H=\Sp(1)/\Z_n$ are given by $\phi_j(g\H)=-jgj\H$ and $\phi_k(g\H)=-kgk\H$; recall that $j,k\in\N(\H)$. Furthermore, $\phi_j$ and $\phi_k$ fix the geodesic $\gamma(r)$ pointwise and their linearizations at any such point are the linear transformations on $T_{\gamma(r)}M_n$ with matrices $(\phi_j)_*=\diag(1,-1,1,-1)$ and $(\phi_k)_*=\diag(1,-1,-1,1)$ with respect to the frame $\left\{\tfrac{\partial}{\partial r},X_1,X_2,X_3\right\}$.
In particular, these linear transformations are orthogonal with respect to any metric of the form \eqref{eq:diagonalmetricstationary}, including $\ggz$. It thus follows that $\phi_j$ and $\phi_k$ are isometries of $(M_n,\ggz)$. Indeed, given any $p\in M_n$, there exist $g_p,g'\in\Sp(1)$ such that $g_p\cdot p$ lies in $\gamma$ and $\phi_j(g_p\cdot p)=(g')^{-1}\phi_j(p)$, so one may write $\phi_j(p)=g'\cdot\phi_j(g_p\cdot p)$ as a composition of diffeomorphisms whose linearizations are isometric, and analogously for $\phi_k$.

Similarly to the previous example, we claim that if $\g$ is any $\Sp(1)$-invariant Riemannian metric on $M_n$ such that $\phi_j$ and $\phi_k$ are isometries, then $\g$ must also be of the form \eqref{eq:diagonalmetricstationary}. Indeed, using $\phi_j$ and $\phi_k$, one can produce sufficiently many isometries of $(M_n,\g)$ so that an argument analogous to \eqref{eq:isometry-trick} may be carried~out.

\begin{remark}
When we make reference to \emph{the Grove-Ziller metric} on $S^2\times S^2$ or $\C P^2\#\overline{\C P^2}$, we mean a Grove-Ziller metric $\ggz$ on any of the (infinitely many) cohomogeneity manifolds $M_n$ where $n$ has the appropriate parity.
\end{remark}

\section{Evolution under Ricci flow}\label{sec:flow}

In this section, we analyze the Ricci flow evolution of the cohomogeneity one $4$-manifolds with $\sec\geq0$ discussed above, showing that the diagonal Ansatz \eqref{eq:diagonalmetricstationary} is preserved (Proposition~\ref{prop:diagonalpreserved}), computing explicitly the Ricci flow equations \eqref{eq:RF} for such metrics (Proposition~\ref{prop:RFequations}) and proving the Theorem in the Introduction.

\subsection{Flow behavior}
As a consequence of uniqueness of the solution to the Ricci flow on a closed manifold $(M,\g_0)$, all isometries of $(M,\g_0)$ remain isometries of $(M,\g_t)$ for all $t>0$. It is actually also known that the isometry group of $(M,\g_t)$ remains constant, that is, no other isometries are created in finite time, as a consequence of backwards uniqueness~\cite{kotschwar}. In particular, cohomogeneity one metrics evolve via Ricci flow through other metrics invariant under the same cohomogeneity one action. Nevertheless, the horizontal geodesic $\gamma$ joining the singular orbits, and hence the description \eqref{eq:cohom1metric} of the cohomogeneity one metric, may in general change with time. We now show that this is not the case for the Grove-Ziller metrics in the $4$-dimensional examples discussed above, using their additional isometries.

\begin{proposition}\label{prop:diagonalpreserved}
The Ricci flow evolution $\g(t)$ of the metric $\ggz=\g(0)$ on each of $S^4$, $\C P^2$, $S^2\times S^2$, and $\C P^2\#\overline{\C P^2}$, is through other diagonal metrics
\begin{equation}\label{eq:diagonalmetricmoving}
\g(t) = \zeta(r,t)^2\dd r^2 + \varphi(r,t)^2 \dd x_1^2 + \psi(r,t)^2 \dd x_2^2 + \xi(r,t)^2 \dd x_3^2, \quad 0<r<L,
\end{equation}
along the $\ggz$-geodesic $\gamma(r)$, where $\zeta$, $\varphi$, $\psi$, and $\xi$, are smooth functions of $r$ and $t$.
\end{proposition}

\begin{proof}
The metric $\ggz$ is a diagonal metric of the form \eqref{eq:diagonalmetricstationary}, and $\gamma(r)$ is a $\ggz$-geodesic parametrized by arclength.
Since isometries are preserved, the Ricci flow evolution of $\ggz$ is through metrics $\g(t)$ which are invariant under the $\G$-action as well as under \eqref{eq:c} on $\C P^2$ and \eqref{eq:phis} on $S^2\times S^2$ and $\C P^2\#\overline{\C P^2}$. As discussed in Subsections \ref{subsec:s4}, \ref{subsec:cp2}, and \ref{subsec:mn}, by means of these isometries, the frame $\left\{\tfrac{\partial}{\partial r},X_1,X_2,X_3\right\}$ along $\gamma(r)$ must be $\g(t)$-orthogonal. In particular, $\g(t)$ are diagonal cohomogeneity one metrics of the form \eqref{eq:diagonalmetricmoving} along $\gamma(r)$, which is $\g(t)$-orthogonal to the $\G$-orbits and hence a horizontal $\g(t)$-geodesic (up to reparametrization).
\end{proof}

\begin{remark}
The Grove-Ziller metric $\ggz$ is smooth but \emph{not real-analytic}, as there are points where all derivatives of $\varphi$, $\psi$, and $\xi$ vanish, but these functions are not globally constant. However, the metrics $\g(t)$, $t>0$, are real-analytic by Bando~\cite{bando}. 
Moreover, since real-analyticity is preserved under Ricci flow, there does not exist a solution to the \emph{backwards Ricci flow} with $\ggz$ as terminal condition.
\end{remark}

\begin{proposition}\label{prop:RFequations}
Let $(M,\g)$ be a $4$-manifold with a cohomogeneity one action of a Lie group $\G$ whose Lie algebra is isomorphic to $\mathfrak{su}(2)$. Assume that $\g$ is a diagonal metric of the form \eqref{eq:diagonalmetricstationary} and that its Ricci flow evolution $\g(t)$ is through other diagonal metrics, as in \eqref{eq:diagonalmetricmoving}. Then the functions $\zeta(r,t)$, $\varphi(r,t)$, $\psi(r,t)$, and $\xi(r,t)$ satisfy the degenerate parabolic system of partial differential equations
\begin{equation}\label{eq:RFcohom1}
\begin{aligned}
\zeta_t &= -\left(\frac{\varphi_r}{\varphi}+\frac{\psi_r}{\psi}+\frac{\xi_r}{\xi}\right)\frac{\zeta_r}{\zeta^2}+\left(\frac{\varphi_{rr}}{\varphi}+\frac{\psi_{rr}}{\psi}+\frac{\xi_{rr}}{\xi}\right)\frac{1}{\zeta}\\[0.1cm]
\varphi_t &=\frac{1}{\zeta^2}\,\varphi_{rr}+\frac{1}{\zeta\psi\xi}\left(\frac{\psi\xi}{\zeta}\right)_{\!r}\varphi_r-\frac{2}{\psi^2\xi^2}\varphi^3+\frac{2(\psi^2-\xi^2)^2}{\psi^2\xi^2}\frac{1}{\varphi}\\[0.1cm]
\psi_t &=\frac{1}{\zeta^2}\,\psi_{rr}+\frac{1}{\zeta\varphi\xi}\left(\frac{\varphi\xi}{\zeta}\right)_{\!r}\psi_r-\frac{2}{\varphi^2\xi^2}\psi^3+\frac{2(\varphi^2-\xi^2)^2}{\varphi^2\xi^2}\frac{1}{\psi}\\[0.1cm]
\xi_t &=\frac{1}{\zeta^2}\,\xi_{rr}+\frac{1}{\zeta\varphi\psi}\left(\frac{\varphi\psi}{\zeta}\right)_{\!r}\xi_r-\frac{2}{\varphi^2\psi^2}\xi^3+\frac{2(\varphi^2-\psi^2)^2}{\varphi^2\psi^2}\frac{1}{\xi}
\end{aligned}
\end{equation}
where subscripts denote derivative with respect to that variable.
\end{proposition}
 
\begin{proof}
The Ricci tensor of \eqref{eq:diagonalmetricmoving} is diagonal on the frame $\left\{\tfrac{\partial}{\partial r},X_1,X_2,X_3\right\}$. It can be computed using \cite[Prop.\ 1.14]{gz-inventiones}, the structure constants of $\mathfrak{su}(2)$, and replacing $\frac{\partial}{\partial r}$ with $\frac{1}{\zeta}\frac{\partial}{\partial r}$ to account for the $\g(t)$-arclength parameter of $\gamma(r)$ for $t>0$, resulting:
\begin{equation*}
\begin{aligned}
\Ric_{\g(t)}\!\left(\tfrac{\partial}{\partial r}, \tfrac{\partial}{\partial r}\right) &= -\frac{\varphi_{rr}\zeta-\varphi_r\zeta_r}{\varphi\zeta^2}-\frac{\psi_{rr}\zeta-\psi_r\zeta_r}{\psi\zeta^2}-\frac{\xi_{rr}\zeta-\xi_r\zeta_r}{\xi\zeta^2}\\
\Ric_{\g(t)}(X_1, X_1) &= \frac{2\varphi^{4}-2(\psi^2-\xi^2)^2}{\psi^{2}\xi^{2}} - \frac{\varphi_r\varphi \psi_r \xi+\varphi_r\varphi \xi_r \psi}{\psi\xi\zeta^{2}} - \frac{\varphi_{rr}\varphi\zeta-\varphi_r\varphi\zeta_r}{\zeta^{3}}
\end{aligned}
\end{equation*}
and expressions analogous to the latter in the directions $X_2$ and $X_3$. Since metric $\g(t)$ and its Ricci tensor $\Ric_{\g(t)}$ are diagonal in the same basis (by Proposition~\ref{prop:diagonalpreserved}), the system \eqref{eq:RFcohom1} is obtained equating the corresponding diagonal entries of $\frac{\partial\g}{\partial t}$ and $-2\Ric_{\g(t)}$. 
\end{proof}

\begin{remark}\label{rem:diagonalpreservation}
A natural question is whether the hypothesis that $\g(t)$ retains the diagonal form \eqref{eq:diagonalmetricmoving} is necessary. For instance, proving existence of solutions to \eqref{eq:RFcohom1} with the appropriate smoothness (boundary) conditions, would, by uniqueness of solutions to Ricci flow,  imply that the diagonal Ansatz is preserved. Nevertheless, these translate into overdetermined boundary conditions for \eqref{eq:RFcohom1}, and determining well-posedness seems to be beyond the reach of standard methods.
\end{remark}

\subsection{Curvature evolution}
We are now ready to analyze the evolution of sectional curvatures of $\ggz$ under Ricci flow, proving the Theorem.

\begin{proof}[Proof of Theorem]
Let $M$ be any of the cohomogeneity one $4$-manifolds discussed in Section~\ref{sec:examples}, and equip it with the Grove-Ziller metric $\ggz$. By Propositions~\ref{prop:diagonalpreserved} and \ref{prop:RFequations}, the Ricci flow evolution of $\g(0)=\ggz$ is through other diagonal metrics of the form \eqref{eq:diagonalmetricmoving}, satisfying \eqref{eq:RFcohom1}.

The initial metric $\g(0)$ is such that, near each singular orbit $B_\pm$, the two functions among $\varphi$, $\psi$, and $\xi$ corresponding to the two noncollapsing directions among $X_1$, $X_2$, and $X_3$ are equal and constant. Up to relabeling, assume these are $X_1$ and $X_2$ near $B_-$, so that 
\begin{equation}\label{eq:constants}
\varphi(r,0)=\psi(r,0)=const.>0, \quad \mbox{for all} \quad 0<r<\varepsilon,
\end{equation}
while $\xi(0,t)=0$ for all $t\geq0$.
Fix $0<r_0<\varepsilon$ and let $\sigma\subset T_{\gamma(r_0)}M$ be the tangent plane spanned by $\frac{\partial}{\partial r}$ and $X_1$. The sectional curvature of $\sigma$ is given by
\begin{equation*}
\sec_{\g(t)}(\sigma)=-\frac{1}{\varphi\zeta}\left(\frac{\varphi_r}{\zeta}\right)_{\!r} =\frac{\varphi_r\zeta_r}{\varphi\zeta^3} -\frac{\varphi_{rr}}{\varphi\zeta^2}\\
\end{equation*}
computed at $r=r_0$.
As a consequence of \eqref{eq:constants}, this plane $\sigma$ is flat at time $t=0$.
Moreover, as $\zeta(r,0)\equiv 1$, we have that
\begin{equation}\label{eq:ddtsec}
\frac{\dd}{\dd t}\sec_{\g(t)}(\sigma)\Big|_{t=0} = -\frac{\varphi_{rrt}}{\varphi}\Big|_{r=r_0,t=0}.
\end{equation}

The evolution equation for $\varphi$ in \eqref{eq:RFcohom1} simplifies enormously due to \eqref{eq:constants}, yielding
\begin{equation*}
\varphi_t\big|_{t=0} = \frac{2(\psi^2-\xi^2)^2 -2\varphi^4}{\varphi\psi^2\xi^2}, \qquad 0<r<\varepsilon.
\end{equation*}
Differentiating the above expression in $r$ twice and using \eqref{eq:constants} once more, we have
\begin{equation*}
\varphi_{rrt}\big|_{r=r_0,t=0} = \frac{4(\xi_r^2 + \xi_{rr}\xi)}{\varphi^3}\Big|_{r=r_0,t=0}.
\end{equation*}
Up to a constant (determined by the ineffective kernel of the action of $X_3$ on the normal disk to $B_-$), the function $\xi(r,t)$ is the length of $\frac{\partial}{\partial \theta}$ for a rotationally symmetric metric $\zeta(r,t)^2\dd r^2+\xi(r,t)^2\dd\theta^2$ on the normal disk to $B_-$ at $x_-$. Thus, by the smoothness conditions for such a metric, $\xi_r=\zeta$ at $r=0$ and $\xi$ must be an odd function of $r$; in particular, $\xi_{rr}(0,t)=0$. Therefore, up to choosing an even smaller $0<r_0<\varepsilon$, we have $\xi_r^2(r_0,0)>0$, while both $\xi(r_0,0)$ and $\xi_{rr}(r_0,0)$ are arbitrarily close to $0$.
It hence follows that \eqref{eq:ddtsec} is strictly negative, so $\sec_{\g(t)}(\sigma)<0$ for all $t>0$ sufficiently small, concluding the proof.
\end{proof}

\begin{remark}\label{rem:final}
More can be said about the evolution of sectional curvatures on the manifolds discussed in the above proof. First, the tangent plane $\sigma$ at $\gamma(r_0)$ could have instead been chosen as the plane spanned by $\frac{\partial}{\partial r}$ and any linear combination of the noncollapsing directions $X_1$ and $X_2$. Of course, a similar situation also takes place near the other singular orbit $B_+$. This means there is a circle's worth of initially flat planes at each point near a singular orbit that become negatively curved for small $t>0$. As a side note, these tangent planes actually integrate to totally geodesic flat strips in $(M,\ggz)$ with an arrangement in the regular part of $M$ reminiscent of an open book decomposition, where the binding is any horizontal geodesic and the ($2$-dimensional) pages are flat strips. These are the so-called \emph{Perelman flat strips}, constructed in the proof of the Soul Conjecture~\cite{perelman-soul}, on either ``half" of $(M,\ggz)$, i.e., on either convex side of a totally geodesic principal orbit.

The behavior of some of these flat planes is the opposite near the middle of $(M,\ggz)$, where they immediately acquire positive curvature for any $t>0$ small. We warn the reader that, in the derivation of formula \eqref{eq:ddtsec} for $\frac{\dd}{\dd t}\sec_{\g(t)}(\sigma)\big|_{t=0}$, we made extensive use of \eqref{eq:constants}, so this expression is not valid on the entire length of $\gamma(r)$, in particular in the latter region. However, the above claim can be verified with an argument similar to \cite[Sec.\ 4.4]{bettiol-pams} using that these planes are tangent to totally geodesic flats, which implies that $\int_{\gamma} \frac{\dd}{\dd t}\sec_{\g(t)}(\gamma'\wedge X)\big|_{t=0}=0$, where $X$ is a vertical direction along $\gamma$ that does not collapse at either singular orbit.
\end{remark}

\end{document}